\documentclass[a4paper,12pt]{article}
\usepackage{enumerate}
\usepackage{blkarray}
\usepackage{amsmath,amssymb,amsthm, mathrsfs}
\usepackage{latexsym,graphicx,kotex}
\usepackage{graphics}
\usepackage{hyperref}
\usepackage[bf, small]{titlesec}
\usepackage{authblk} 
\usepackage{soul}
\usepackage[displaymath,mathlines]{lineno}
\usepackage{cite}
\numberwithin{equation}{section}

\theoremstyle{plain}
\newtheorem{Thm}{Theorem}[section]
\newtheorem{Lem}[Thm]{Lemma}
\newtheorem{Prop}[Thm]{Proposition}
\newtheorem{Cor}[Thm]{Corollary}

\theoremstyle{definition}

\usepackage{standalone}
\usepackage{tikz}
\usetikzlibrary{arrows,positioning,matrix,fit,backgrounds,shapes,shapes.geometric, calc, intersections}

\usetikzlibrary{shapes.callouts,decorations.pathmorphing} 
\usepackage{calc,stackrel} 
\usepackage{geometry}
\usetikzlibrary{matrix,positioning,decorations.pathreplacing,decorations.markings}
\tikzstyle{vertex}=[circle, draw, inner sep=0pt, minimum size=6pt] 
\newcommand{\vertex}{\node[vertex]}
\usetikzlibrary{arrows,matrix} 

\title{On $(1,2)$-step competition graphs of multipartite tournaments}

\author[]{Myungho Choi
\thanks{Corresponding author\\ 
E-mail addresses: nums8080@snu.ac.kr(M.Choi), srkim@snu.ac.kr(S.-R.Kim)}}
\author[]{Suh-Ryung Kim
}

\affil[]{Department of Mathematics Education,
Seoul National University, Seoul 08826, Republic of Korea}

 \newcounter{statement}
\newcommand{\statement}[2]{%
 \begin{equation}\refstepcounter{statement}\tag{S\thestatement}\label{#1}%
  \parbox{\dimexpr\linewidth-4em}{#2}%
 \end{equation}%
}

\begin{document}
\maketitle
\begin{abstract}
 A multipartite tournament is an orientation of a complete $k$-partite graph for some positive integer $k\geq 3$.
We say that a multipartite tournament $D$ is tight if every partite set forms a clique in the $(1,2)$-step competition graph, denoted by $C_{1,2}(D)$, of $D$.
In this paper, we completely characterize $C_{1,2}(D)$ for a tight multipartite tournament $D$.
We will study $C_{1,2}(D)$ for a multipartite tournament $D$ that is not tight in a follow up paper.
\end{abstract}


\section{Introduction}
In this paper, all the graphs and digraphs are assumed to be finite and simple. (For all undefined
graph theory terminologies, see \cite{bondy}.)
Given a digraph $D$ and a vertex $v$ of $D$,
we define $N^+(v)=\{u \in V(D) \mid (v,u) \in A(D)\}$, $N^-(v)=\{u \in V(D) \mid (u,v) \in A(D) \}$, $d^+(v)=|N^+(v)|$, and $d^-(v)=|N^-(v)|$.
We use the expression $u \to v$ when $(u,v) \in A(D)$. 
When representing negation, add a slash ($/$) to the symbol.

For vertices $x$ and $y$ in a digraph $D$, $d_D(x,y)$ denotes the number of arcs in a shortest directed path from $x$ to $y$ in $D$ if it exists.
For positive integers $i$ and $j$, the {\it $(i,j)$-step competition graph } of a digraph $D$, denoted by $C_{i,j}(D)$, is a graph on $V(D)$ where $uv \in E(C_{i,j}(D))$ if and only if there exists a vertex $w$ distinct from $u$ and $v$ such that (a) $d_{D-v}(u,w) \leq i$ and $d_{D-u}(v,w) \leq j$ or (b) $d_{D-u}(v,w) \leq i$ and $d_{D-v}(u,w) \leq j$.
 If two vertices of a digraph $D$ are adjacent in $C_{1,2}(D)$, then we just say that they are adjacent in the rest of this paper.
 
The $(1,1)$-step competition graph of a digraph $D$ is the {\it competition graph} of $D$.
Given a digraph $D$, the competition graph of $D$, denoted by $C(D)$, is the graph having the vertex set $V(D)$ and the edge set $\{uv \mid u \to w, v\to w \text{ for some } w \in V(D) \}$.
Cohen~\cite{cohen1968interval} introduced the notion of competition graph while studying
predator–prey concepts in ecological food webs. Cohen’s empirical observation that real-world competition graphs are
usually interval graphs had led to a great deal of research on the structure of competition graphs and on the relation between
the structure of digraphs and their corresponding competition graphs. In the same vein, various variants of competition graph
have been introduced and studied, one of which is the notion of $(i, j)$-step competition introduced by Factor and Merz~\cite{factor20111}. For recent work on this topic, see
\cite{factor20111,kamibeppu2012sufficient,
kim2015generalization,kuhl2013transversals,li2012competition,
mckay2014competition,zhang20161,zhang2013note}.
Choi \emph{et al.}~\cite{choi20171} studied the $(1,2)$-step competition graph of an orientation of a complete bipartite graph.
In this paper, we study the $(1, 2)$-step competition graph of an orientation of a complete $k$-partite graph for some integer $k \geq 3$.

For a digraph $D$, we say that vertices $u$ and $v$ in $D$ $(1,2)$-{\it compete} provided there exists a vertex $w$ distinct from $u$, $v$ that satisfies one of the following:
\begin{itemize}
\item{} there exist an arc $(u,w)$ and a directed $(v,w)$-path of length $2$ not traversing $u$;
\item{} there exist a directed $(u,w)$-path of length $2$ not traversing $v$ and an arc $(v,w)$.
\end{itemize}
We call $w$ in the above definition a {\it $(1,2)$-step common out-neighbor }of $u$ and $v$.
 It is said that two vertices {\it compete} in $D$ if they have a common out-neighbor in $D$. 
 If $u$ and $v$ compete or $(1,2)$-compete in $D$,
 then we say that $u$ and $v$
 $\{1,2\}$-{\it compete} in $D$. 
 We also say that a set is {\it $\{1,2\}$-competing} (resp.\ anti-$\{1,2\}$-competing) if any pair of vertices in the set $\{1,2\}$-competes (resp.\ if no two vertices in the set $\{1,2\}$-compete) in $D$.

We note that a vertex set $S$ of a digraph $D$ is a $\{1,2\}$-competing set (resp.\ an anti-$\{1,2\}$-competing set) if and only if $S$ is a clique (resp.\ a stable set) in $C_{1,2}(D)$.
Here, a \emph{stable set} of a graph is a set of vertices no two of which are adjacent.
 
 Given vertex sets $X$ and $Y$ of a graph $G$,
 we use the symbol $X \sim Y$ if for each $x\in X$ and $y\in Y-\{x\}$, $x$ and $y$ are adjacent in $G$.
When $X$ or $Y$ is a singleton, the set brackets are not used. 
For example, we write $ u \sim Y$ if $X=\{u\}$. 

 If $u$ and $v$ are adjacent in $C_{1,2}(D)$ by competing (resp.\ $(1,2)$-competing), then it is simply represented as $u \sim_1 v$ (resp.\ $u \sim_{1,2} v$).
 We note that $u \sim v$ if and only if $u\sim_1 v$ or $u\sim_{1,2}v$.

We call an orientation of a complete $k$-partite graph for some positive integer $k$ a \emph{$k$-partite tournament}. 
In this paper, a \emph{multipartite tournament} refers to a $k$-partite tournament with $k \ge 3$.
A tournament of order $n$ may be regarded as an $n$-partite tournament for some positive integer $n$.

In this paper, we study the $(1,2)$-step competition graph of a multipartite tournament $D$ and we completely characterize $C_{1,2}(D)$ when 
 each partite set of $D$ is $\{1,2\}$-competing
according to the size of a maximum anti-$\{1,2\}$-competing set (Theorems~\ref{thm:cha-complete}, \ref{thm:two-stable-set-intersect}, and \ref{Thm:maximum-I=3}). 
 For simplicity, we call a multipartite tournament with every partite set $\{1,2\}$-competing a {\it tight} multipartite tournament.

\section{Preliminaries}\label{property}
Let $D$ be a digraph.
We call a vertex with outdegree $0$ in $D$ a {\em sink}.
It is obvious that
each non-sink vertex has at least one out-neighbor in $D$.
For a non-sink vertex $u$ and a vertex $v$,
$u \stackrel{*}{\to} v$ means that $v$ is the only out-neighbor of $u$ in $D$.

\begin{Prop}\label{prop:same-partite}
Let $D$ be a multipartite tournament, and $u$ and $v$ be two non-sink vertices belonging to the same partite set in $D$. 
Then the following are true:
\begin{itemize}
\item[(1)] $u$ and $v$ do not compete if and only if $N^+(u) \cap N^+(v) = \emptyset$;
\item[(2)] $u$ and $v$ do not $(1,2)$-compete if and only if $N^+(u) \cup N^+(v) \subseteq X$ for some partite set $X$ of $D$.
\item[(3)] $u$ and $v$ are not adjacent in $C_{1,2}(D)$ if and only if $N^+(u) \cap N^+(v) = \emptyset$ and $N^+(u) \cup N^+(v) \subseteq X$ for some partite set $X$ of $D$.
\end{itemize}

\end{Prop}
\begin{proof}
(1) and (2) are obvious by the definitions of `compete' and `$(1,2)$-compete'.
(3) is an immediate consequence of (1) and (2).
\end{proof}

\begin{Prop} \label{prop:one-degree}
Let $D$ be a digraph.
Suppose $u \stackrel{*}{\to} v$ for some vertices $u$ and $v$ in $D$.
Then $u$ and $v$ are not adjacent in $C_{1,2}(D)$.
\end{Prop}
\begin{proof}
Since $u \stackrel{*}{\to} v$,
each directed path from $u$ to a vertex $z$ in $V(D)-\{u,v\}$ must traverse $v$ and so there is no directed walk from $u$ to $z$ in $D-v$. 
Therefore $u \not \sim v$.
\end{proof}

\begin{Prop} \label{Prop:different-partite-adjacent}
Let $D$ be a multipartite tournament and $u$ and $v$ be two non-sink vertices with $(u,v)\in A(D)$ belonging to the distinct partite sets in $D$.
Then the following are equivalent:
\begin{itemize}
\item[(i)] $u$ and $v$ are not adjacent in $C_{1,2}(D)$;
\item[(ii)]
either
$u \stackrel{*}{\to} v$, or $N^+(u) \cap N^+(v) = \emptyset$ and there exists a partite set $X$ such that $N^+(v) \subseteq X$ and $N^+(u) \subseteq X \cup \{v\}$.
\end{itemize}
\end{Prop}

\begin{proof}
We first show that (i) implies (ii).
Suppose that $u$ and $v$ are not adjacent. Then $u$ and $v$ have no common out-neighbor.
If $u \stackrel{*}{\to} v$, then we are done.
Suppose $u \not \stackrel{*}{\to} v$.
Then $N^+(u) -\{v\} \neq \emptyset$.
Since $v$ is a non-sink vertex and $u \to v$,
$N^+(v) - \{u\} \neq \emptyset$. 
Thus 
there exist an out-neighbor $x$ of $u$ and an out-neighbor $y$ of $v$ distinct from $v$ and $u$, respectively.
If $x$ and $y$ belong to different partite sets, then $(x,y)$ or $(y,x) \in A(D)$ and so $u \sim_{1,2}v$, which is a contradiction.
Therefore $x$ and $y$ belong to the same partite set.
Since $x$ and $y$ are arbitrarily chosen,
we conclude that
$N^+(v) \subseteq X$ and $N^+(u) \subseteq X\cup \{v\}$ for some partite set $X$.

Now we show that (ii) implies (i).
If  $u \stackrel{*}{\to} v$
then $u$ and $v$ are not adjacent by Proposition~\ref{prop:one-degree}.
Suppose that  $u \not \stackrel{*}{\to} v$,
$N^+(u) \cap N^+(v)=\emptyset$, and
$N^+(v) \subseteq X$, and $N^+(u) \subseteq X\cup \{v\}$ for some partite set $X$.
Then
 \begin{equation} \label{eq:Lem:different-partite-adjacent} (N^+(u) \cup N^+(v)) - \{u,v\} \subseteq X. \end{equation}
Suppose, to the contrary, that $u$ and $v$ are adjacent. Then, since $N^+(u) \cap N^+(v)=\emptyset$,
$u \sim_{1,2} v$.
Let $w$ be a $(1,2)$-step common out-neighbor of $u$ and $v$.
We first consider the case where there exist a directed $(u,w)$-path $P= u \to u' \to w$ for some vertex $u'$ distinct from $v$ and an arc $(v,w)$.
Then $u'$ (resp.\ $w$) is an out-neighbor of $u$ (resp.\ $v$) distinct from $u$ and $v$.
 Therefore $\{u',w\}\subseteq X$ by \eqref{eq:Lem:different-partite-adjacent}.
 However, $(u',w)$ is an arc on $P$ and so we reach a contradiction.
Thus $u$ and $v$ are not adjacent.
Since we did not use the arc $(u,v)$ in the previous argument, we are still able to apply it to the case where there exist a directed $(v,w)$-path $P= v \to v' \to w$ for some vertex $v'$ distinct from $u$ and an arc $(u,w)$.
\end{proof}

The following corollary is immediately true by Propositions~\ref{prop:same-partite} and \ref{Prop:different-partite-adjacent}.
\begin{Cor} \label{Cor:cor-adajcent}
Let $D$ be a multipartite tournament and $u$, $v$ be two non-sink vertices in $D$.
Then $u$ and $v$ are adjacent in $C_{1,2}(D)$ if the following happen:
\begin{enumerate}[{(1)}]
\item $v$ is not the only out-neighbor of $u$;
\item $u$ is not the only out-neighbor of $v$;
\item for any partite set $X$ with $N^+(v) \subseteq X$ (resp.\ $N^+(u) \subseteq X$),
 $N^+(u)\not\subseteq X \cup \{v\}$ (resp.\ 
 $N^+(v)\not\subseteq X \cup \{u\}$).
\end{enumerate}
\end{Cor}

\begin{Lem}\label{lem:sink}
Let $D$ be a tight multipartite tournament.
If $D$ has a sink, then
$C_{1,2}(D)$ is isomorphic to $K_{|V(D)|-1}$ with an isolated vertex.
\end{Lem}
\begin{proof}
	 Suppose that $D$ has a sink $v$. 
	 Then $v$ constitutes a trivial partite set of $D$ and so $v$ is an out-neighbor of the vertices in $V(D)\setminus \{v\}$.
	 Thus the statement is true.
\end{proof}

 \begin{Lem}
 Let $D$ be a tight multipartite tournament.
Then every anti-$\{1,2\}$-competing set has size at most three.
 \end{Lem}
 \begin{proof}
 Suppose that $D$ has an anti-$\{1,2\}$-competing set $S$.
 If $D$ has a sink, then $|S| = 2$ by Lemma~\ref{lem:sink}.
 Now suppose $D$ has no sinks.
Let $X_1,\ldots,X_k$ be the partite sets of $D$ and \[\Lambda=\{i \mid S \cap X_i \neq \emptyset\}.\]
Suppose $|\Lambda| \geq 4$.
We take four vertices in distinct partite sets of $D$. Then they induce the tournament $T$ of order $4$, so there exists a pair of vertices competing in $T$ since $T$ has four vertices and six arcs.
Therefore $|\Lambda| \leq 3 $.
Thus $|S| \leq 3$.
 \end{proof}

 By the above lemma,
each anti-$\{1,2\}$-competing set of a tight multipartite tournament has size one, two, or three.
In each case of the sizes one, two, and three,
we characterize the $(1,2)$-step competition graph of a tight multipartite tournament.

   An anti-$\{1,2\}$-competing set in a digraph $D$
 is \emph{maximum} if $D$ contains no larger anti-$\{1,2\}$-competing set.
 
\section{$C_{1,2}(D)$ for a tight multipartite tournament $D$ with a maximum anti-$\{1,2\}$-competing set of size one}
If $D$ is a multipartite tournament with a maximum anti-$\{1,2\}$-competing set of size one, then $C_{1,2}(D)$ is complete.
In this section, we characterizes the sizes of partite sets of multipartite tournaments whose $(1,2)$-step competition graphs are complete.

Since a tournament of order $k \geq 5$ may be considered as a $k$-partite tournament,
the following is true by Corollary~\ref{Cor:cor-adajcent}.
\begin{Cor} \label{Cor:outdegree-complete}
Let $D$ be a tournament with at least five vertices.
If each vertex in $D$ has outdegree at least two, then $C_{1,2}(D)$ is complete.
\end{Cor}

A tournament $D$ is {\it regular} provided all vertices in $D$ have the same out-degree.
We say that $D$ is {\it near regular} provided the largest difference between the out-degrees of any two vertices is $1$.
It is a well-known fact that, for each positive integer $n$, there exists a regular tournament of order $n$ if $n$ is odd and a near regular tournament when if $n$ is even.
Since a regular or near regular tournament with at least five vertices has minimum outdegree at least two,
the following is immediately true by Corollary~\ref{Cor:outdegree-complete}.
\begin{Lem} \label{lem:existence-tournament}
For $n \geq 5$, there exists a tournament of order $n$ whose $(1,2)$-step competition graph is complete.
\end{Lem}

Given a graph $G$, two vertices $u$ and $v$ of $G$ are said to be \emph{true twins} if they have the same closed neighborhood.
We may introduce an analogous notion for a digraph.
Given a digraph $D$, two vertices $u$ and $v$ of $D$ are said to be \emph{true twins} if they have the same open out-neighborhood and open in-neighborhood.

Let $D$ be a digraph. If there is a directed path of length $2$ from a vertex $x$ to a vertex $y$ in $D$, we call $y$ a \emph{$2$-step out-neighbor} of $x$.

\begin{Lem} \label{lem:true-twin}
If two non-sink vertices are true twins in a digraph $D$, then they are true twins in $C_{1,2}(D)$.
\end{Lem}

\begin{proof}
Suppose that there exist two vertices $u$ and $v$ which are true twins in $D$.
Since $D$ is loopless, there is no arc between $u$ and $v$, that is,
\begin{equation} \label{eq:lem:true-twin_1} u \not \to v \quad \text{and} \quad  v \not \to u.	
\end{equation}
Since $u$ is a non-sink vertex,
$u$ has an out-neighbor $x$.
Then $x\neq v$ by \eqref{eq:lem:true-twin_1}.
Since $u$ and $v$ are true twins,
$x$ is also an out-neighbor of $v$.
Then $x \neq u$ by \eqref{eq:lem:true-twin_1}.
Therefore $u \sim_1 v$.

Take a vertex $w$ distinct from $v$ among the vertices adjacent to $u$ in $C_{1,2}(D)$.
If $u \sim_1 w$, then there exists a common out-neighbor of $u$ and $w$ distinct from $v$ by \eqref{eq:lem:true-twin_1} and so the common out-neighbor is also a common out-neighbor of $v$ and $w$, which implies $v \sim_1 w$.
Suppose $u \sim_{1,2} w$.
Then $u$ and $w$ have a $(1,2)$-step common out-neighbor $y$ distinct from $u$ and $w$.
Therefore (a) $w \to y$ and $u\to a \to y$ for some vertex $a$ with $a\neq w$ or (b) $u \to y$ and $w\to b \to y$ for some vertex $b$ with $b\neq u$.

{\it Case 1}. $y=v$.
Then, by \eqref{eq:lem:true-twin_1}, (b) cannot hold and so (a) holds.
 Since $u$ and $v$ are true twins,
$w \to u$ and $v\to a$.
Thus $w \to u \to a$ and $v\to a$.
Then, since $D$ is loopless, $a\neq v$ and so $v \sim_{1,2} w$.

{\it Case 2}. $y\neq v$.
If (a) holds, then $v \to a \to y$ since $u$ and $v$ are true twins, which implies $v \sim_{1,2} w$.
Now suppose that (b) holds.
Since $u$ and $v$ are true twins, $v \to y$.
Then, since $b\neq v$, $y$ is a $(1,2)$-step common out-neighbor of $v$ and $w$.
Therefore, in both cases, $v \sim_{1,2}w$ and so the statement is true.
\end{proof}

\begin{Lem} \label{lem:making-multi-complete}
Let $k$ be a positive integer with $k \geq 3$; $n_1,\ldots, n_k$ be positive integers such that $n_1 \geq \cdots \geq n_k$;  $n'_1, \ldots ,n'_k$ be positive integers such that
$n'_1 \geq \cdots \geq n'_k$, $n'_1 \geq n_1$,
$n'_2 \geq n_2,\ldots$, and $n'_k \geq n_k$.
If $D$ is an orientation of $K_{n_1,\ldots,n_k}$ whose $(1,2)$-step competition graph is complete,
then there exists an orientation $D'$ of $K_{n'_1,\ldots,n'_k}$ whose $(1,2)$-step competition graph is complete.
\end{Lem}

\begin{proof}
Suppose that $D$ is an orientation of $K_{n_1,\ldots,n_k}$ whose $(1,2)$-step competition graph is complete.
Let $X_1,X_2,\ldots,X_k$ be the partite sets of $D$  satisfying $|X_i|=n_i$ for each $1 \leq i \leq k$.
Then we construct an orientation of $K_{n'_1,n_2,\ldots,n_k}$ whose $(1,2)$-step competition graph is complete in the following way.
If $n'_1 = n_1$, then we take $D$ as a desired orientation.
Suppose $n'_1 > n_1$.
Then we add a new vertex $v$ to $X_1$ and an arc $(v,x)$ for each out-neighbor $x$ of some vertex $u$ in $X_1$ to obtain a digraph $H_1$ so that
\[
 A(D)\subset A(H_1), \quad \text{and} \quad  N^+_D(u)= N^+_{H_1}(u) = N^+_{H_1}(v). \]
Therefore $N^-_D(u) =N^-_{H_1}(u)= N^-_{H_1}(v)$ and so $u$ and $v$ are true twins in $H_1$.
Since $C_{1,2}(D)$ is complete and $|V(D)| \geq 2$,
$N^+_{D}(u) \neq \emptyset$.
Therefore $C_{1,2}(H_1)$ is complete by Lemma~\ref{lem:true-twin}.
We may repeat this process until we obtain a desired orientation $H_{n'_1-n_1}$.
Inductively, we obtain an orientation $H_{t}$ of $K_{n'_1,\ldots,n'_k}$ whose $(1,2)$-step competition graph is complete where $t=(n'_1+\cdots+n'_k)-(n_1+\cdots+n_k)$.
Therefore the statement is true.
\end{proof}

The following theorem characterizes the sizes of partite sets of multipartite tournaments whose $(1,2)$-step competition graphs are complete.

\begin{Thm} \label{thm:cha-complete}
Let $k$ be a positive integer with $k \geq 3$ and $n_1,n_2,\ldots,n_k$ be positive integers such that $n_1 \geq \cdots \geq n_k$.
There exists an orientation $D$ of $K_{n_1,n_2,\ldots,n_k}$ whose $(1,2)$-step competition graph is complete if and only if one of the following holds.
\begin{itemize}
\item[(a)] $k=3$, and (i) $n_2 \geq 3$ and $n_3=1$ or (ii) $n_3 \geq 2$.
\item[(b)] $k=4$, and (i) $n_1 \geq 3$ and $n_2=1$ or (ii) $n_2 \geq 2$.
\item[(c)] $k \geq 5$.
\end{itemize}
\end{Thm}

\begin{proof}
We first show the ``only if" part.
Suppose that there exists an orientation $D$ of $K_{n_1,n_2,\ldots,n_k}$ whose $(1,2$)-step competition graph is complete.
Then, since $k \geq 3$, $|V(D)| \geq 3$ and so each vertex has outdegree at least $1$ in $D$.
If there exists a vertex $v$ of outdegree $1$ in $D$, then there exists a vertex nonadjacent to $v$ in $C_{1,2}(D)$ by Proposition~\ref{prop:one-degree}.
Therefore \begin{equation} \label{eq:thm:cha-complete}
d^+(v) \geq 2 \end{equation} for each vertex $v$ in $D$.
Thus \[2|V(D)| \leq |A(D)|.\]
Let $X_1,\ldots,X_k$ be the partite sets of $D$ satisfying $|X_i|=n_i$ for each $1\leq i \leq k$.

Suppose $k=3$.
Then, if $n_2=1$, then $|V(D)|=n_1+2$ and so $|A(D)| = 2 n_1 + 1$, which contradicts $2|V(D)| \leq |A(D)|$.
Therefore $n_2 \geq 2$.
To show by contradiction, suppose $n_2=2$ and $n_3 =1$.
Let $X_2=\{v_1,v_2\}$, $X_3=\{v_3\}$.
Then
each vertex in $X_1$ is not a common out-neighbor of two vertices in $X_2 \cup X_3$ by~\eqref{eq:thm:cha-complete}.
Therefore each pair of $\{v_1,v_3\}$ and $\{v_2,v_3\}$ has a $(1,2)$-step common out-neighbor in $D$.

Let $u$ be a $(1,2)$-step common out-neighbor of $v_1$ and $v_3$.
Then $u \in N^+(v_1)$ or $u\in N^+(v_3)$.
Suppose $u \in N^+(v_1)$. Then $u \in X_1$ and there exists a $(v_3,u)$-directed path $P$ of length $2$ not traversing $v_1$.
Thus the interior point on the directed path must be $v_2$ and so $(v_2,u)\in A(D)$.
Hence $u$ has outdegree at most one, a contradiction to \eqref{eq:thm:cha-complete}.
Therefore $u\in N^+(v_3)$.
Then $u$ is a $2$-step out-neighbor of $v_1$.
If $u \in X_1$, then each $(v_1,u)$-directed path of length $2$ must traverse $v_3$ and so $v_1$ and $v_3$ cannot have a $(1,2)$-step common out-neighbor, a contradiction.
Therefore $u \notin X_1$.
By symmetry, any $(1,2)$-step common out-neighbor of $v_2$ and $v_3$ does not belong to $X_1$.
Thus $u =v_2$ and $v_1$ is the only $(1,2)$-step common out-neighbor of $v_2$ and $v_3$.
Hence $v_1$ and $v_2$ must be $2$-step out-neighbors of $v_2$ and $v_1$, respectively, and so out-neighbors of $v_3$.
Therefore $N^+(v_1) \cup N^+(v_2) \subseteq X_1$.
Thus $v_1$ and $v_2$ do not $(1,2)$-compete and so have a common out-neighbor $x$.
Then $x\in X_1$ and $d^+(x) \leq 1$, which contradicts~\eqref{eq:thm:cha-complete}.
Thus
$n_2\geq 3$ or $n_3\geq 2$ and so (a) holds.

Suppose $k=4$ and $n_2=1$.
Then $|V(D)|=n_1+3$ and
$|A(D)|=3n_1+3$.
By \eqref{eq:thm:cha-complete},
$2|V(D)|=2(n_1+3)  \leq |A(D)|$ and so
$n_1 \geq 3$.
Therefore (b) holds. Thus we have shown that the ``only if" part is true.

Now we show the ``if" part.

{\it Case 1}. $k =3$ or $4$.
We consider orientations $D_1$, $D_2$, $D_3$, and $D_4$ of $K_{3,3,1}$, $K_{2,2,2}$, $K_{3,1,1,1}$, and $K_{2,2,1,1}$, respectively, given in Figure~\ref{fig:tournaments-complete}  whose $(1,2)$-step competition graphs are complete.
By applying to Lemma~\ref{lem:making-multi-complete} to $D_1$, $D_2$, $D_3$, and $D_4$, respectively,
we may obtain an orientation of $K_{n_1,n_2,\ldots,n_k}$ whose $(1,2)$ competition graph is complete when
(a) $k=3$, and (i) $n_2 \geq 3$ and $n_3=1$ or (ii) $n_3 \geq 2$;
(b) $k=4$, and (i) $n_1 \geq 3$ and $n_2=1$ or (ii) $n_2 \geq 2$.

{\it Case 2}. $k \geq 5$.
We obtain a tournament $D$ of order $k$ whose $(1,2)$-step competition graph is complete by Lemma~\ref{lem:existence-tournament}.
Then, by applying to Lemma~\ref{lem:making-multi-complete} to $D$,
we may obtain an orientation of $K_{n_1,n_2,\ldots,n_k}$ whose $(1,2)$-step competition graph is complete.
Therefore the ``if'' part is true.
\end{proof}

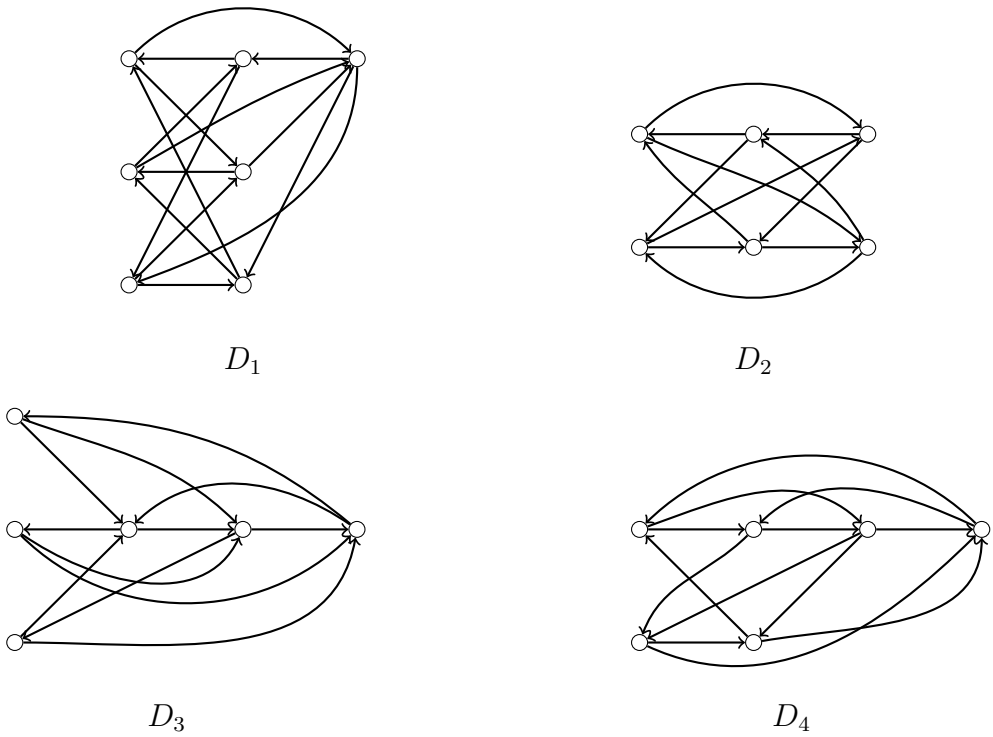
\begin{figure}
\begin{center}
\begin{tikzpicture}[x=1.0cm, y=1.0cm]
   \vertex (x1) at (0,0) [label=above:$$]{};
   \vertex (x2) at (0,1.5) [label=above:$$]{};
   \vertex (x3) at (0,3) [label=above:$$]{};
   \vertex (x4) at (1.5,0) [label=above:$$]{};
   \vertex (x5) at (1.5,1.5) [label=above:$$]{};
   \vertex (x6) at (1.5,3) [label=above:$$]{};
   \vertex (x7) at (3,3) [label=above:$$]{};
   \path[->,thick]
   (x3) edge [out=45,in=135] (x7)
   (x3) edge (x5)
   (x2) edge (x6)
   (x2) edge [out=30,in=-160] (x7)
   (x1) edge (x5)
   (x1) edge (x4)
   (x6) edge (x3)
   (x6) edge (x1)
   (x5) edge (x7)
   (x5) edge  (x2)
   (x4) edge (x2)
   (x4) edge  (x3)
   (x7) edge  (x4)
   (x7) edge (x6)
   (x7) edge  [out=-90,in=20] (x1)
	;
\draw (1.5, -1) node{$D_1$};
\end{tikzpicture}
\qquad \qquad \qquad \qquad
\begin{tikzpicture}[x=1.0cm, y=1.0cm]
   \vertex (x1) at (0,0) [label=above:$$]{};
   \vertex (x2) at (0,1.5) [label=above:$$]{};
   \vertex (x3) at (1.5,0) [label=above:$$]{};
   \vertex (x4) at (1.5,1.5) [label=above:$$]{};
   \vertex (x5) at (3,0) [label=above:$$]{};
   \vertex (x6) at (3,1.5) [label=above:$$]{};
   \path[->,thick]
   (x1) edge (x3)
   (x3) edge  (x5)
   (x5) edge  [out=-135,in=-45] (x1)
   (x4) edge (x2)
   (x6) edge  (x4)
   (x2) edge [out=45,in=135] (x6)
   (x2) edge  [out=-25,in=140] (x5)
   (x4) edge  (x1)
   (x6) edge (x3)
   (x1) edge (x6)
   (x3) edge [out=135,in=-55] (x2)
   (x5) edge [out=120,in=-35] (x4)
	;
\draw (1.5, -1.5) node{$D_2$};
\end{tikzpicture}

\begin{tikzpicture}[x=1.0cm, y=1.0cm]

   \tikzset{middlearrow/.style={decoration={
  markings,
  mark=at position 0.95 with
  {\arrow{#1}},
  },
  postaction={decorate}
  }
  }

   \vertex (x1) at (0,0) [label=above:$$]{};
   \vertex (x2) at (0,1.5) [label=above:$$]{};
   \vertex (x3) at (0,3) [label=above:$$]{};
   \vertex (x4) at (1.5,1.5) [label=above:$$]{};
   \vertex (x5) at (3,1.5) [label=above:$$]{};
   \vertex (x6) at (4.5,1.5) [label=above:$$]{};
 \path[->,thick]
   (x3) edge  (x4)
  (x3) edge [out=-20,in=135](x5)
   (x2) edge [out=-35,in=-115](x5)
   (x2) edge [out=-45,in=-135](x6)
   (x1) edge (x4)
   (x1) edge [out=0,in=-100] (x6)
   (x4) edge (x2)
   (x4) edge (x5)
   (x5) edge (x1)
   (x5) edge (x6)
   (x6) edge [out=145,in=45] (x4)
   (x6) edge [out=140,in=0] (x3)
	;
\draw (2, -1) node{$D_3$};
\end{tikzpicture}
\qquad \qquad \qquad \qquad
\begin{tikzpicture}[x=1.0cm, y=1.0cm]
   \vertex (x1) at (0,0) [label=above:$$]{};
   \vertex (x2) at (0,1.5) [label=above:$$]{};
   \vertex (x3) at (1.5,0) [label=above:$$]{};
   \vertex (x4) at (1.5,1.5) [label=above:$$]{};
   \vertex (x5) at (3,1.5) [label=above:$$]{};
   \vertex (x6) at (4.5,1.5) [label=above:$$]{};
   \path[->,thick]
   (x2) edge (x4)
   (x2) edge [out=20,in=135] (x5)
   (x1) edge (x3)
   (x1) edge [out=-25,in=-135] (x6)
   (x3) edge (x2)
   (x3) edge [out=10,in=-90] (x6)
   (x4) edge [out=-135,in=70] (x1)
   (x4) edge (x5)
   (x5) edge (x3)
   (x5) edge (x6)
   (x5) edge (x1)
   (x6) edge [out=135,in=45](x2)
   (x6) edge [out=155,in=45](x4)
	;
\draw (2, -1) node{$D_4$};
\end{tikzpicture}
\end{center}
\caption{ $D_1$, $D_2$, $D_3$, and $D_4$ are orientations of $ K_{3,3,1},K_{2,2,2}, K_{3,1,1,1}$, and $K_{2,2,1,1}$, respectively, whose $(1,2)$-step competition graphs are complete}
\label{fig:tournaments-complete}
\end{figure}

\section{$C_{1,2}(D)$ for a tight multipartite tournament $D$ with a maximum anti-$\{1,2\}$-competing set of size two}
\begin{Lem} \label{lem:diameter_two}
Let $D$ be a tight multipartite tournament.
Suppose that $D$ has no sinks and an anti-$\{1,2\}$-competing set $S=\{u_1,u_2\}$.
If $C_{1,2}(D)-\{u_1,u_2\}$ is non-complete,
then $u_1\stackrel{*}{\to} u_2$ or $u_2\stackrel{*}{\to} u_1$.
\end{Lem}
\begin{proof}
Since $S$ is an anti-$\{1,2\}$-competing set and every partite set in $D$ is a $\{1,2\}$-competing set,
$u_1$ and $u_2$ belong to distinct partite sets.
Without loss of generality, we may assume that
$X_1$ and $X_2$ are partite sets of $D$ containing $u_1$ and $u_2$, respectively, and
\[u_1 \to u_2.\]
Suppose that $C_{1,2}(D)-\{u_1,u_2\}$ is non-complete. 
Then there exist two nonadjacent vertices $v_1$ and $v_2$ in $C_{1,2}(D)-\{u_1,u_2\}$.
 To show $u_1 \stackrel{*}{\to} u_2$, we suppose $u_1 \not \stackrel{*}{\to} u_2$.
Then $u_2$ is not the only out-neighbor of $u_1$.
Consequently, by Proposition~\ref{Prop:different-partite-adjacent}, there exists a partite set $X$ such that $\emptyset\neq N^+(u_2)\subseteq X$ and $\emptyset\neq N^+(u_1)\subseteq X\cup \{u_2\}$.
Then $X\neq X_1$ and $X\neq X_2$.
Thus
$ N^+(u_1) \cap X_2 =\{u_2\} $
and $N^+(u_2)\cap X_1 = \emptyset$.

Now we show that each of $v_1$ and $v_2$ has $u_1$ or $u_2$ as an out-neighbor.
If $v_1 \in X_1$, then $u_2$ is an out-neighbor of $v_1$ since $N^+(u_2)\cap X_1 = \emptyset$.
If $v_1 \in X_2$, then $u_1$ is an out-neighbor of $v_1$ since $N^+(u_1) \cap X_2 =\{u_2\}$.
If $v_1 \notin X_1 \cup X_2$,
then at least one of $u_1$ and $u_2$ is an out-neighbor of $v_1$ since $u_1$ and $u_2$ have no common out-neighbor.
Therefore $v_1$ has $u_1$ or $u_2$ as an out-neighbor.
By symmetry, $v_2$ has $u_1$ or $u_2$ as an out-neighbor.
Thus $v_1 \sim v_2$ in $C_{1,2}(D)-\{u_1,u_2\}$ and we reach a contradiction. Hence $u_1 \stackrel{*}{\to} u_2.$
\end{proof}

\begin{Lem} \label{lem:picky}
Let $D$ be a tight multipartite tournament.
Suppose that $D$ has no sinks and there are two vertices $u_1$ and $u_2$ such that $u_1\stackrel{*}{\to} u_2$.
Then for any anti-$\{1,2\}$-competing set $S$ of size $2$ with $S \cap \{u_1,u_2\}=\emptyset$,
\begin{enumerate}[{(1)}]
\item $v_1\stackrel{*}{\to} v_2$
where $v_1$ and $v_2$ are appropriate labels for the vertices of $S$;
\item $v_1$ and $u_1$ belong to the same partite set of $D$, and $v_2$ and $u_2$ belong to distinct partite sets of $D$.
\end{enumerate}
\end{Lem}

\begin{proof}
Let $X_1$ and $X_2$ be partite sets of $D$ containing $u_1$ and $u_2$, respectively.
Then $N^-(u_1) = V(D) - (X_1 \cup \{u_2\})$ since $u_1\stackrel{*}{\to} u_2$.
Suppose that $D$ has an anti-$\{1,2\}$-competing set $S$ of size $2$ with $S \cap \{u_1,u_2\} = \emptyset$.
Let $S=\{v_1,v_2\}$.
If $S \cap X_1= \emptyset$, then $u_1$ is a common out-neighbor of $v_1$ and $v_2$, a contradiction.
Therefore $S \cap X_1 \neq \emptyset$.
Then, since every partite set of $D$ is $\{1,2\}$-competing, not both $v_1$ and $v_2$ belong to $X_1$.
Without loss of generality, we may assume $v_1 \in X_1$ and $v_2 \notin X_1$.
Then, $u_1$ and $v_1$ belong to the same partite set, so they $\{1,2\}$-compete.
Since $v_2 \neq u_2$, $u_1$ is an out-neighbor of $v_2$.
If $v_1$ has an out-neighbor $v'$ distinct from $v_2$, then $v_1 \to v' \to u_1$ and so $u_1$ is a $(1,2)$-step common out-neighbor of $v_1$ and $v_2$, a contradiction.
Therefore
\[v_1\stackrel{*}{\to} v_2\]
since $v_1$ is not a sink.
Then, since $u_1\stackrel{*}{\to} u_2$ and $u_2 \neq v_2$,
$u_1 \not \sim_1 v_1$ and so $u_1 \sim_{1,2} v_1$.
We also note that 
$N^+(u_1) \subseteq X_2$ and $N^+(v_1)\subseteq X$ where $X$ is the partite set containing $v_2$.
Then, by Proposition~\ref{prop:same-partite}(2),
$X \neq X_2$.
Therefore we have shown that the parts (1) and (2) are valid.
\end{proof}
The {\it complement }of a graph $G$ is a graph $\overline{G}$ on the same vertices such that two distinct vertices of $\overline{G}$ are adjacent if and only if they are not adjacent in $G$.
A tree containing exactly two vertices of degree at least two is called a {\it double-star}.
A {\it caterpillar} is a tree in which all the vertices are within distance $1$ of a central path.
\begin{Thm} \label{thm:two-stable-set-intersect}
Let $D$ be a tight multipartite tournament with a maximum anti-$\{1,2\}$-competing set of size two.
Then the complement of $C_{1,2}(D)$ is one of the following types:
\begin{enumerate}
\item[A.] a star graph with or without isolated vertices;
\item[B.] a double-star graph with or without isolated vertices;
\item[C.] a disjoint union of at least two star graphs with or without isolated vertices;
\item[D.] a caterpillar which has at least one vertex of degree $2$ with or without isolated vertices.
\end{enumerate}
\end{Thm}

\begin{proof}
We denote $C_{1,2}(D)$ by $G$.
If $D$ has a sink, then $G$ is isomorphic to $K_{|V(D)|-1}$ with an isolated vertex by Lemma~\ref{lem:sink} and so the complement $\overline{G}$ of $G$ is of Type A.

We assume that $D$ has no sinks.
Let $t$ be the maximum number of disjoint anti-$\{1,2\}$-competing sets with size $2$ of $D$ and let
$\mathcal{S}= \{S_1,S_2,\ldots,S_t\}$
where $S_i$ is an anti-$\{1,2\}$-competing set with size $2$ of $D$ for each $1\leq i \leq t $ and $S_1,\ldots,S_t$ are mutually disjoint.
In addition, let \[S_1=\{u_1,u_2\}\] and $X_1$ and $X_2$ be the partite sets of $D$ containing $u_1$ and $u_2$, respectively.
Without loss of generality, we may assume
\[u_1\to u_2.\]
By the way, since $\{u_1,u_2\}$ is a maximum anti-$\{1,2\}$-competing set of size two by the hypothesis,
each vertex in $G-\{u_1,u_2\}$ is adjacent to $u_1$ or $u_2$.
Thus $\overline{G}$ is of Type A or Type B if $G-\{u_1,u_2\}$ is a complete graph.

Now we suppose that
$G-\{u_1,u_2\}$ is not a complete graph.
Then, by Lemma~\ref{lem:diameter_two},
\[u_1 \stackrel{*}{\to} u_2.\]
Now, we let \[S_2=\{v_1,v_2\}.\]
Then $\{v_1,v_2\} \cap \{u_1,u_2\}=\emptyset$.

For simplicity, we say that an anti-$\{1,2\}$-competing set $\{x_1,x_2\}$ with $x_1\stackrel{*}{\to} x_2$ is {\it picky}.
Then $\{u_1,u_2\}$ is picky.
Therefore
$\{v_1,v_2\}$ is also picky by Lemma~\ref{lem:picky}(1).
Without loss of generality, we may assume \[v_1\stackrel{*}{\to} v_2.\]
Then $v_1 \in X_1$ and $v_2 \notin X_2$
by Lemma~\ref{lem:picky}(2).
We may assume $v_2 \in X_3$ where $X_3$ is a partite set of $D$.
Then, since $u_1\stackrel{*}{\to} u_2$ and $v_1\stackrel{*}{\to} v_2$,
any pair of vertices in $V(D)- X_1$ distinct from $\{u_2,v_2\}$ has a common out-neighbor $u_1$ or $v_1$.
Moreover, since every partite set of $D$ is $\{1,2\}$-competing, $X_1$ forms a clique in $G$.
Therefore
\statement{sta:thm:two-stable-set-intersect_1}{any anti-$\{1,2\}$-competing set of size $2$ distinct from $\{u_2,v_2\}$ intersects with both $X_1$ and $V(D)- X_1$.}
{\it Case 1}.
 $u_2 \sim v_2$.
Then, by \eqref{sta:thm:two-stable-set-intersect_1},
\statement{sta:thm:two-stable-set-intersect_2}{any anti-$\{1,2\}$-competing set of size $2$ intersects with both $X_1$ and $V(D)- X_1$.}

{\it Subcase 1}. $t \geq 3$ (recall that $t$ is the maximum number of disjoint anti-$\{1,2\}$-competing sets with size two of $D$).
Then $S_3$ is picky by
Lemma~\ref{lem:picky}(1).
We denote $S_3$ by $\{s_{3,1},s_{3,2}\}$
with $s_{3,1}\stackrel{*}{\to} s_{3,2}$.
Since $S_1$ is picky and $S_1\cap S_3= \emptyset$, $s_{3,1} \in X_1$ and $s_{3,2} \notin X_2$ by Lemma~\ref{lem:picky}(2).
Since $S_2$ is also picky and $S_2 \cap S_3= \emptyset$,
 $s_{3,2} \notin X_3$ by the same lemma.
   Therefore we may assume $s_{3,2}\in X_4$ where $X_4$ is a partite set of $D$.
 Inductively, we may let $S_i=\{s_{i,1},s_{i,2}\}$ so that

 \begin{equation}\label{eq:thm:two-stable-set-intersect_1}
 s_{i,1}\stackrel{*}{\to} s_{i,2}, \quad s_{i,1}\in X_1, \quad \text{and}\quad s_{i,2} \in X_{i+1}
 \end{equation}
 where $X_{i+1}$ is a partite set of $D$ for each $1\leq i\leq t$.
 
 If every anti-$\{1,2\}$-competing set of $D$ belongs to $\mathcal S$, then
 $\overline{G}$ is of Type C.
 Now suppose that
 there exists an anti-$\{1,2\}$-competing set $\{y_1,y_2\}$ of $D$ with $\{y_1,y_2\} \notin \mathcal{S}$.
Then, by the maximality of $\mathcal{S}$,
there exists a pair $\{s_{j,1},s_{j,2}\}$ in $\mathcal{S}$ such that $\{s_{j,1},s_{j,2} \} \cap \{y_1,y_2\} \neq \emptyset$ for some $j \in \{1,\ldots,t\}$.
Without loss of generality, we may assume $y_1 \in X_1$ and $y_2 \in V(D) - X_1$ by \eqref{sta:thm:two-stable-set-intersect_2}.
We first suppose $\{s_{j,1},s_{j,2} \} \cap \{y_1,y_2\}=\{s_{j,1}\}$, that is, $y_1=s_{j,1}$ and $y_2 \neq s_{j,2}$.
Then $y_2$ has at least $t-1$ out-neighbors in $T=\{s_{1,1},s_{2,1},\ldots,s_{t,1}\}$ by \eqref{eq:thm:two-stable-set-intersect_1}.
In addition, for each $1 \leq i \leq t$ except $i=j$, there exists a directed path $P_i:=s_{j,1} \to s_{j,2}\to s_{i,1}$ which does not traverse $y_2$.
Therefore $
s_{1,1},s_{2,1},\ldots,s_{t,1}$ except $s_{j,1}$ are $2$-step out-neighbors of $s_{j,1}$ obtained by $P_1,\ldots,P_t$.
Since $|N^+(y_2)\cap T|\geq t-1$,
$|N^+(y_2)\cap (T-\{s_{j,1}\})|\geq t-2$.
Then, since $t\geq 3$, $|N^+(y_2)\cap (T-\{s_{j,1}\})|\geq 1$.
Therefore
there exists a $(1,2)$-step common out-neighbor of $s_{j,1}$ and $y_2$ in $T-\{s_{j,1}\}$, a contradiction for $\{y_1,y_2\}$ being an anti-$\{1,2\}$-competing set.
Thus $\{s_{j,1},s_{j,2} \} \cap \{y_1,y_2\}=\{s_{j,2}\}$, that is, $y_1 \neq s_{j,1} $ and $y_2 = s_{j,2}$.

We will claim that $s_{j,2}$ is the only vertex in $\{s_{1,2},s_{2,2},\ldots,s_{t,2}\}$ which is not adjacent to $y_1$ in $G$.
Since $t \geq 3$,
there exist two vertices $s_{j_1,1}$ and $s_{j_2,1}$ for some $j_1,j_2 \in \{1,\ldots,t\}- \{j\}$ and, by~\eqref{eq:thm:two-stable-set-intersect_1},
every vertex in $V(D)-X_1$ except $s_{j_1,2}$ (resp.\ $s_{j_2,2}$) is an in-neighbor of $s_{j_1,1}$ (resp.\ $s_{j_2,1}$).
Thus $s_{j_1,1}$ and $s_{j_2,1}$
are out-neighbors of $s_{j,2}$ and so
$s_{j,2}$ has an out-neighbor distinct from $y_1$ in $X_1$.
Since $D$ has no sinks,
$y_1$ has an out-neighbor.
To show $y_1 \stackrel{*}{\to} s_{j,2}$,
suppose $y_1$ has an out-neighbor $y'$ distinct from $s_{j,2}$.
Then $y' \notin X_1$.
Let $s'$ be an out-neighbor of $s_{j,2}$ distinct from $y_1$ in $X_1$.
Then $y' \to s'$ or $s' \to y'$. 
Therefore $s'$ is a $(1,2)$-step common out-neighbor of $y_1$ and $s_{j,2}$ and so $y_1 \sim s_{j,2}$, which is impossible.
Hence $y_1 \stackrel{*}{\to} s_{j,2}$.
Since $s_{j,1}\stackrel{*}{\to} s_{j,2}$ by~\eqref{eq:thm:two-stable-set-intersect_1},
 $s_{j,1}$ is a common out-neighbor of $s_{1,2},s_{2,2},s_{3,2},\ldots,s_{t,2}$ except $s_{j,2}$.
 Then, since $y_1 \neq s_{j,1}$,
$s_{j,2}$ is a $(1,2)$-step common out-neighbor of $y_1$ and $s_{i,2}$ for each $1\leq i \leq t$ except $i=j$.
Thus $s_{j,2}$ is the only vertex in $\{s_{1,2},s_{2,2},\ldots,s_{t,2}\}$ which is not adjacent to $y_1$ in $G$.
Hence we may conclude that, for every edge $e$ except the edges $s_{1,1}s_{1,2},\ldots,s_{t,1}s_{t,2}$ in $\overline{G}$,
exactly one vertex in $\{
s_{1,2},s_{2,2},\ldots,s_{t,2}\}$ is incident to $e$ in $\overline{G}$.
This implies that
$\overline{G}$ is a disjoint union of $t$ star graphs whose centers are $s_{1,2},s_{2,2},\ldots,s_{t,2}$ with or without isolated vertices and so is of Type C.

{\it Subcase 2.} $t=2$.
If $\{u_1,u_2\}$ and $\{v_1,v_2\}$ are the only anti-$\{1,2\}$-competing sets in $D$,
then $\overline{G}$ is of Type C.
Suppose that there exists an anti-$\{1,2\}$-competing set $\{w_1,w_2\}$ in $D$ distinct from $\{u_1,u_2\}$ and $\{v_1,v_2\}$.
Then, since $t=2$, $\{w_1,w_2\} \cap \{u_1,u_2,v_1,v_2\} \neq \emptyset$.
Without loss of generality, we may assume $w_1 \in X_1$ and $w_2 \in V(D) - X_1$ by \eqref{sta:thm:two-stable-set-intersect_2}.
We first suppose $\{w_1,w_2\} \subset \{u_1,u_2,v_1,v_2\}$.
Then, either $w_1=u_1$ and $w_2=v_2$ or $w_1=v_1$ and $w_2=u_2$.
We consider the case where $w_1=u_1$ and $w_2=v_2$.
Then $u_1$ and $v_2$ are not adjacent.
If $v_2 \to u_2$, then $u_2$ is a common out-neighbor of $u_1$ and $v_2$, a contradiction.
Then, since $v_2$ and $u_2$ belong to distinct partite sets in $D$, $u_2 \to v_2$.
Thus $v_2$ is a common out-neighbor of $v_1$ and $u_2$ and so $v_1$ and $u_2$ are adjacent.
In case $w_1=v_1$ and $w_2=u_2$, we may conclude that $u_1$ and $v_2$ are adjacent by a similar argument.
Hence we have shown that
\statement{sta:thm:two-stable-set-intersect_3}{if $\{w_1,w_2\} \subset \{u_1,u_2,v_1,v_2\}$, then
exactly one of $\{u_1,v_2\}$ and $\{u_2,v_1\}$ is an anti-$\{1,2\}$-competing set in $D$.}

Now we suppose $\{w_1,w_2\} \not \subset \{u_1,u_2,v_1,v_2\}$.
Then  $|\{w_1,w_2\} \cap \{u_1,u_2,v_1,v_2\} |=1$ since $t=2$.
To the contrary, suppose $\{w_1,w_2\} \cap \{u_1,u_2,v_1,v_2\} =\{w_1\}$.
Then, since $w_1\in X_1$, $w_1=u_1$ or $w_1=v_1$.
Without loss of generality, we may assume $w_1=u_1$.
We note that $u_1 \stackrel{*}{\to} u_2$ and $v_1 \stackrel{*}{\to} v_2$.
Then $v_1$ is an out-neighbor of $u_2$.
Since $w_2 \in V(D)- X_1$ and $w_2 \neq v_2$, $v_1$ is an out-neighbor of $w_2$.
Then, since $w_2 \neq u_2$, $v_1$ is a $(1,2)$-step common out-neighbor of $u_1$ and $w_2$, a contradiction.
Therefore \[\{w_1,w_2\} \cap \{u_1,u_2,v_1,v_2\} =\{w_2\}. \]
Then
\[w_2=u_2 \quad \text{or}\quad w_2=v_2.\]
In addition,  \[w_1 \neq u_1 \quad \text{and} \quad w_1 \neq v_1.\]
Thus $v_1\in N^+(u_2)- \{w_1\}$ and $u_1 \in N^+(v_2)- \{w_1\}$ (recall that $v_1 \in N^+(u_2)$ and $u_1 \in N^+(v_2)$).
Since $w_1$ is a non-sink vertex, $N^+(w_1)\neq \emptyset$.
Then, since $u_2$ and $v_2$ are distinct,
$N^+(w_1) - \{u_2\}\neq \emptyset$ or $N^+(w_1) - \{v_2\}\neq \emptyset$.
Take a vertex $w' $ in $(N^+(w_1) - \{u_2\}) \cup (N^+(w_1) - \{v_2\})$.
Since $w_1 \in X_1$, $w' \in V(D)-X_1$.
If $w' \neq v_2$ (resp.\ $w' \neq u_2$), then $w' \to v_1$ (resp.\ $w' \to u_1$) and so $v_1$ (resp.\ $u_1$) is a $(1,2)$-step common out-neighbor of $u_2$ (resp.\ $v_2$) and $w_1$. 
Therefore $w_1 \sim u_2$ or $w_1 \sim v_2$.
Then, since $w_2=u_2$ or $w_2=v_2$,
exactly one of $\{w_1,u_2\}$ and $\{w_1,v_2\}$ is an anti-$\{1,2\}$-competing set in $D$.
Hence we have shown that
\statement{sta:thm:two-stable-set-intersect_4}
{if $\{w_1,w_2\} \not \subset \{u_1,u_2,v_1,v_2\}$, then (1) $\{w_1,w_2\} \cap \{u_1,u_2,v_1,v_2\} =\{w_2\}\subsetneq \{u_2,v_2\}$ and (2) exactly one of $\{w_1,u_2\}$ and $\{w_1,v_2\}$ is an anti-$\{1,2\}$-competing set in $D$.
}
Noting that $\{w_1,w_2\}$ was arbitrarily chosen, we may argue as follows.
Suppose that not both end points of $e$ are contained in $\{u_1,u_2,v_1,v_2\}$ for any edge $e$ except $u_1u_2$ and $v_1v_2$ in $\overline{G}$.
This case corresponds to \eqref{sta:thm:two-stable-set-intersect_4}.
Now fix an edge $e$ distinct from  $u_1u_2$ and $v_1v_2$.
Then, by \eqref{sta:thm:two-stable-set-intersect_4}(1), exactly one end point $u$ of $e$ belong to $\{u_1,u_2,v_1,v_2\}$ and $u=u_2$ or $u=v_2$.
 Let $u'$ be the other end point of $e$.
 If $u'$ is adjacent to $u_1$ or $v_1$, then edge $u'u_1$ or edge $u'v_1$ violates \eqref{sta:thm:two-stable-set-intersect_4}(1).
 Therefore $u'$ is adjacent to neither $u_1$ nor $v_1$.
 In addition, by \eqref{sta:thm:two-stable-set-intersect_4}(2), if $u=u_2$ (resp.\ $u=v_2$), $u'$ cannot be adjacent to $v_2$ (resp.\ $u_2$).
Thus $\overline{G}$ is of Type C.
Now suppose that there exists an edge that is not $u_1u_2$ or $v_1v_2$ and whose end points are contained in $ \{u_1,u_2,v_1,v_2\}$.
Then the subgraph of $\overline{G}$ induced by $\{u_1,u_2,v_1,v_2\}$ is an induced path of length $3$ by \eqref{sta:thm:two-stable-set-intersect_3}.
If there exists an edge $e$ except $u_1u_2$ and $v_1v_2$ in $\overline{G}$ such that not both end points of $e$ are contained in $\{u_1,u_2,v_1,v_2\}$, then, by \eqref{sta:thm:two-stable-set-intersect_4}, (1) exactly one end point $u$ of $e$ must be either $u_2$ or $v_2$, (2) the other end point $u'$ of $e$ cannot belong to $\{u_1,u_2,v_1,v_2\}$, and (3) $u$ is the only vertex in $\{u_1,u_2,v_1,v_2\}$ adjacent to $u'$.
Thus
 $u_1$ or $v_1$ is a vertex of degree $2$ in $\overline{G}$.
Hence $\overline{G}$ is of Type D.

{\it Case 2}. $u_2 \not \sim v_2$.
Suppose, to the contrary, that there exists an anti-$\{1,2\}$-competing set $\{w_1,w_2\}$ in $V(D) - \{u_2,v_2\}$.
Then, by \eqref{sta:thm:two-stable-set-intersect_1},
we may assume $w_1 \in X_1$ and $ w_2\in V(D)- X_1$.
Since $w_2 \neq v_2$, $v_1$ is an out-neighbor of $w_2$ in $D$.
If $w_1 = u_1$, then $v_1$ is a $(1,2)$-step common out-neighbor of $w_1$ and $w_2$ since $u_1 \to u_2 \to v_1$ is a directed path in $D$, a contradiction.
Therefore $w_1 \neq u_1$.
Then, since $w_2 \neq u_2$, $\{w_1,w_2\}$ is an anti-$\{1,2\}$-competing set in $V(D)- \{u_1,u_2\}$, so, by Lemma~\ref{lem:picky}, $w_1\stackrel{*}{\to} w_2$.
Therefore $w_1$ is a common out-neighbor of $u_2$ and $v_2$ in $D$, which is a contradiction to the case assumption.
Thus there exists no anti-$\{1,2\}$-competing set of size $2$ in $V(D)- \{u_2,v_2\}$, that is, $G-\{u_2,v_2\}$ is complete.
Thus $\overline{G}$ is of Type A or Type B.
Then, since $u_2 \not \sim v_2$ by the case assumption,
$u_2$ and $v_2$ are adjacent in $\overline{G}$ and so they have degree at least two (recall that $u_1$ is adjacent to $u_2$ and $v_1$ is adjacent to $v_2$ in $\overline{G}$). 
Since $D$ has a maximum anti-$\{1,2\}$-competing set of size two, each vertex in $V(G)-\{u_2,v_2\}$ $\{1,2\}$-competes with $u_2$ or $v_2$ in $D$.
Therefore $\overline{G}$ is of Type B.
\end{proof}

\section{$C_{1,2}(D)$ for a tight multipartite tournament $D$ with a maximum anti-$\{1,2\}$-competing set of size three}
In this section,
we characterize the $(1,2)$-step competition graph of a tight multipartite tournament with a maximum anti-$\{1,2\}$-competing set of size three.

\begin{Lem} \label{lem:cycle-adjacent}
Let $D$ be a digraph having a directed cycle $C$ of order $l$ for some $l \in \{3,4\}$ and $X$ be a subset of $V(D)$.
Suppose that each vertex $u$ in $X - V(C)$ has two out-neighbors $u_1$ and $u_2$ on $C$ such that both $(u_1,u_2)$-section and $(u_2,u_1)$-section of $C$ have length at most $2$.
Then $(X-V(C)) \sim X$. 
\end{Lem}

\begin{proof}
We take a vertex $u$ in $X- V(C)$ and a vertex $v$ in $X$. First we assume $v \in V(C)$. 
Let $u_1$ and $u_2$ be out-neighbors of $u$ on $C$ satisfying the given condition.
Without loss of generality, we may assume that the $(u_1,u_2)$-section of $C$ has length $2$.
Then $u_2$ is a $(1,2)$-step out-neighbor of $u_1$ and $u$.
If $l=3$, then $u_1$ is a common out-neighbor of $u_2$ and $u$.
If $l=4$, then $u_1$ is a $(1,2)$-step common out-neighbor of $u$ and $u_2$.
Any vertex on $C$ other than $u_2$ and $u_1$ shares $u_2$ or $u_1$ as a common out-neighbor with $u$.
Therefore $(X-V(C)) \sim (X\cap V(C))$.

Now we assume $v \in X-V(C)$.
If $u$ and $v$ do not share a common out-neighbor on $C$,
then $l=4$ and $u_2$ is a $(1,2)$-step common out-neighbor of $u$ and $v$.
Therefore $(X-V(C)) \sim (X- V(C))$.
Thus we have shown $(X-V(C)) \sim X$.
\end{proof}

\begin{Thm}\label{Thm:maximum-I=3}
Let $D$ be a tight multipartite tournament of order $n$ with a maximum anti-$\{1,2\}$-competing set $S$ of size three.
Then one of the following is true:
\begin{enumerate}[{(a)}]
\item $C_{1,2}(D) \cong K_{n}-E(K_3)$;
\item $n \geq 4$ and $C_{1,2}(D) \cong K_n - (E(K_3) \cup E(K_{1,l}))$ for a positive integer $l \leq n-3$ where the center of $K_{1,l}$ is a vertex $v$ such that $V(K_3) \cap V(K_{1,l})  =\{v\}$.
\end{enumerate}
\end{Thm}

\begin{proof}
If $D$ has a sink $u$, then $u$ is a common out-neighbor of the vertices in $V(D)-X$ where $X$ is the partite set containing $u$, and so there is no anti-$\{1,2\}$-competing set intersecting with three partite sets.
Thus $D$ has no sink.

Let $S=\{u_1,u_2,u_3\}$ and $X_1,X_2$, and $X_3$ be the partite sets of $D$ with $u_i \in X_i$ for each $1\leq i \leq 3$.
Since $S$ is an anti-$\{1,2\}$-competing set in $D$,
no two vertices in $S$ compete and so
the vertices in $S$ form a directed cycle $C$ of order $3$ in $D$. Without loss of generality, we may assume
\[C:=u_1 \to u_2 \to u_3 \to u_1 \]
{\it Case 1}. $d^+(u_i)=1$ for each $1\leq i  \leq 3$.
Then each vertex in $V(D) - S$ has at least two out-neighbors in $S$.
Therefore $\left( V(D)-S \right) \sim V(D)$ by Lemma~\ref{lem:cycle-adjacent}.
Thus $C_{1,2}(D)$ contains a subgraph isomorphic to $K_n - E(K_3)$.
Then, since $S$ is a stable set in $C_{1,2}(D)$, $C_{1,2}(D) \cong K_n - E(K_3)$.

{\it Case 2}. $d^+(u_j)\neq1$ for some $j \in \{1,2,3\}$.
Then $d^+(u_j) \geq 2$.
Without loss of generality, we may assume $j=1$.
Then, since $u_1\to u_2$ and $u_2 \to u_3$,
\begin{equation}\label{eq:Thm:maximum-I=3}
\emptyset \neq N^+(u_1) - \{u_2\} \subseteq X_3
\end{equation}
by Proposition~\ref{Prop:different-partite-adjacent} and so
\[n \geq 4.\]

We first show that $u_1$ is the only vertex on $C$ of outdegree at least $2$.
Suppose that $u_q$ has outdegree at least $2$ for some $q \in \{2,3\}$.
We first consider the case $q=3$. Then $u_3$ has outdegree at least $2$.
Let $w$ be an out-neighbor of $u_3$ distinct from $u_1$.
Since $u_3 \in X_3$, $w \notin X_3$.
$u_1$ has an out-neighbor $x$ in $X_3$ by \eqref{eq:Thm:maximum-I=3} and there is an arc between $w$ and $x$.
If $w\to x$ (resp.\ $x \to w$), then $x$ (resp.\ $w$) is a $(1,2)$-step common out-neighbor of $u_1$ and $u_3$.
Thus $u_1 \sim u_3$, which is impossible.
Therefore $q=2$.
Then $d^+(u_2) \geq 2$.
Since $u_1 \not \sim u_2$, $ N^+(u_2)=N^+(u_2)- \{u_1\} \subseteq X_3$ by Proposition~\ref{Prop:different-partite-adjacent}.
Hence
there exists a vertex $x \in X_3- \{u_3\}$ belonging to $N^+(u_2)$.
Then $u_1 \to x$ or $x\to u_1$. If $u_1 \to x$, then $x$ is a common out-neighbor of $u_1$ and $u_2$, a contradiction.
Therefore  $x \to u_1$ and so $u_1$ is a $(1,2)$-step common out-neighbor of $u_2$ and $u_3$, a contradiction.
Thus $u_1$ is the only vertex on $C$ of outdegree at least $2$.
Hence
\begin{equation} \label{eq:character-different-partite-11}
 u_2\stackrel{*}{\to} u_3 \quad \text{and} \quad u_3\stackrel{*}{\to} u_1. \end{equation}
We denote $N^+(u_1)-\{u_2\}$ by $N$.
Then $N\neq \emptyset$.
By~\eqref{eq:Thm:maximum-I=3} and~\eqref{eq:character-different-partite-11},
each vertex in $(V(D) - N)- S$ has at least two out-neighbors in $S$.
Since $S=V(C)$ and the length of $C$ is $3$, 
\begin{equation} \label{eq:character-different-partite-13}
\left ((V(D) - N)- S \right) \sim \left( V(D)-N\right)
\end{equation} by Lemma~\ref{lem:cycle-adjacent}.

Take $v \in N$. Then $v \in X_3$ by~\eqref{eq:Thm:maximum-I=3}.
By \eqref{eq:character-different-partite-11}, \[v \to u_2. \] Thus $u_2$ is a common out-neighbor of $u_1$ and $v$ and a $(1,2)$-step common out-neighbor of $v$ and $u_3$.
Therefore $v \sim u_1$ and $v \sim u_3$.
Moreover, each vertex in $V(D) - S$ has an out-neighbor as $u_2$ or $u_3$ by~\eqref{eq:character-different-partite-11}.
Then, since $v\to u_2$, $v \sim \left(V(D)-S\right)$.
Thus $u_2$ is the only vertex that can be nonadjacent to $v$ in $C_{1,2}(D)$.
Since $v$ was chosen from $N$,
\begin{equation} \label{eq:character-different-partite-14}
N_{C_{1,2}(D)}(v)\supseteq V(D)-\{u_2\}
\end{equation}
for any $v \in N$.
Let $M$ be the maximum subset of $N$ such that
\[	 N_{C_{1,2}(D)}(v)= V(D)-\{u_2\}\] 
for any $v \in M$.
 Then, by \eqref{eq:character-different-partite-13} and \eqref{eq:character-different-partite-14}, we may conclude 
 $C_{1,2}(D) \cong K_n -( E(K_3) \cup E(K_{1,l}))$ where $K_3$ and $K_{1,l}$ represents the complete graph with the vertex set $S$ and the complete bipartite graph with the bipartition $(\{u_2\}, M\cup \{u_1,u_3\})$.
Therefore (a) holds if $M=\emptyset$ and (b) holds otherwise.
\end{proof}

\section{Acknowledgement}

This work was supported by Science Research Center Program through the National Research Foundation of Korea (NRF) grant funded by the Korean Government (MSIT) (NRF-2022R1A2C1009648 and 2016R1A5A1008055).

\bibliographystyle{plain}

\begin{thebibliography}{10}

\bibitem{bondy}
John~Adrian Bondy, Uppaluri Siva~Ramachandra Murty, et~al.
\newblock {\em Graph theory with applications}, volume 290.
\newblock Macmillan London, 1976.

\bibitem{choi20171}
Jihoon Choi, Soogang Eoh, Suh-Ryung Kim, and Sojung Lee.
\newblock On (1, 2)-step competition graphs of bipartite tournaments.
\newblock {\em Discrete Applied Mathematics}, 232:107--115, 2017.

\bibitem{cohen1968interval}
Joel~E Cohen.
\newblock Interval graphs and food webs: a finding and a problem.
\newblock {\em RAND Corporation Document}, 17696, 1968.

\bibitem{factor20111}
Kim~AS Factor and Sarah~K Merz.
\newblock The (1, 2)-step competition graph of a tournament.
\newblock {\em Discrete Applied Mathematics}, 159(2-3):100--103, 2011.

\bibitem{kamibeppu2012sufficient}
Akira Kamibeppu.
\newblock A sufficient condition for kim’s conjecture on the competition
  numbers of graphs.
\newblock {\em Discrete Mathematics}, 312(6):1123--1127, 2012.

\bibitem{kim2015generalization}
Suh-Ryung Kim, Jung~Yeun Lee, Boram Park, and Yoshio Sano.
\newblock A generalization of opsut’s result on the competition numbers of
  line graphs.
\newblock {\em Discrete Applied Mathematics}, 181:152--159, 2015.

\bibitem{kuhl2013transversals}
Jaromy Kuhl.
\newblock Transversals and competition numbers of complete multipartite graphs.
\newblock {\em Discrete Applied Mathematics}, 161(3):435--440, 2013.

\bibitem{li2012competition}
Bo-Jr Li and Gerard~J Chang.
\newblock Competition numbers of complete r-partite graphs.
\newblock {\em Discrete Applied Mathematics}, 160(15):2271--2276, 2012.

\bibitem{mckay2014competition}
Brendan~D McKay, Pascal Schweitzer, and Patrick Schweitzer.
\newblock Competition numbers, quasi-line graphs, and holes.
\newblock {\em SIAM Journal on Discrete Mathematics}, 28(1):77--91, 2014.

\bibitem{zhang20161}
Xinhong Zhang and Ruijuan Li.
\newblock The (1, 2)-step competition graph of a pure local tournament that is
  not round decomposable.
\newblock {\em Discrete Applied Mathematics}, 205:180--190, 2016.

\bibitem{zhang2013note}
Xinhong Zhang, Ruijuan Li, Shengjia Li, and Gaokui Xu.
\newblock A note on the existence of edges in the (1, 2)-step competition graph
  of a round digraph.
\newblock {\em Australas. J Comb.}, 57:287--292, 2013.

\end{thebibliography}

\end{document}